\documentclass{amsart}
\usepackage{amsmath,amsthm,amsfonts,amssymb,amstext, latexsym}

\newtheorem{theorem}{Theorem}[section]
\newtheorem{lem}[theorem]{Lemma}
\newtheorem{prop}[theorem]{Proposition}
\newtheorem{cor}[theorem]{Corollary}
\newtheorem{conj}[theorem]{Conjecture}
\theoremstyle{definition}
\newtheorem{definition}[theorem]{Definition}

\newtheorem{condition}[theorem]{Condition}

\theoremstyle{remark}

\numberwithin{equation}{section}

\newcommand{\ad}{\text{ad}}
\newcommand{\Hor}{{\mathcal{H}}}
\newcommand{\V}{{\mathcal{V}}}
\newcommand{\ra}{\rightarrow}

\newcommand{\CP}{{\mathbb{CP}}}

\newcommand{\Ad}{\text{Ad}}
\newcommand{\ddt}{\frac{\partial}{\partial t}}

\newcommand{\lb}{\langle}
\newcommand{\rb}{\rangle}

\newcommand{\mg}{\mathfrak{g}}

\newcommand{\mm}{\mathfrak{m}}
\newcommand{\mU}{\mathfrak{U}}
\newcommand{\mV}{\mathcal{V}}
\newcommand{\mH}{\mathcal{H}}
\newcommand{\ml}{\mathfrak{l}}
\newcommand{\mA}{\mathcal{A}}

\newcommand{\bs}{\backslash}

\newcommand{\R}{\mathbb{R}}

\begin{document}

\title[Totally Geodesic Foliations of Lie Groups]{Totally Geodesic Foliations and Doubly Ruled Surfaces in a  Compact Lie Group}
\author{Marius Munteanu, Kristopher Tapp}

\begin{abstract}
For a Riemannian submersion from a simple compact Lie group with a bi-invariant metric, we prove the action of its holonomy group on the fibers is transitive.  As a step towards classifying Riemannian submersions with totally geodesic fibers, we consider the parameterized surface induced by lifting a base geodesic to points along a geodesic in a fiber.  Such a surface is ``doubly ruled'' (it is ruled by horizontal geodesics and also by vertical geodesics).  Its characterizing  properties allow us to define ``doubly ruled parameterized surfaces'' in any Riemannian manifold, independent of Riemannian submersions.  We initiate a study of the doubly ruled parameterized surfaces in compact Lie groups and in other symmetric spaces by establishing several rigidity theorems and by providing several examples with unexpected properties.\end{abstract}

\address{Department of Mathematics\\SUNY Oneonta, Oneonta, NY 13820}
\email{munteam@oneonta.edu}
\address{Department of Mathematics\\ Saint Joseph's University\\
5600 City Ave.\\
Philadelphia, PA 19131}
\email{ktapp@sju.edu}
%\subjclass{53C}
%\keywords{Riemannian submersion, Lie group, good triple, doubly ruled surface}
%\date{\today}

\maketitle
\date{\today}

%---------------------------------------------------------------------------------

\section{Introduction}
Riemannian submersions are a central tool for many geometric constructions such as nonnegatively curved manifolds and Einstein manifolds, but they are also important for proving rigidity theorems.  For example, the proof of the Diameter Rigidity Theorem for positively curved manifolds~\cite{GG1} hinged on the classification of Riemannian submersions from round spheres (achieved in~\cite{GG2},\cite{Wi}).  This theorem motivated work towards classifying Riemannian submersions from other symmetric spaces.  For example, the Riemannian submersions from flat Euclidean spaces were classified in~\cite{GW}.  Those from $\CP^n$ with totally geodesic fibers had been previously classified in~\cite{Es2}.

K. Grove posed the problem of classifying Riemannian submersions from compact Lie groups with bi-invariant metrics~\cite[Problem 5.4]{Gro}.  The special case of one-dimensional fibers was solved in~\cite{M}.  The general case is likely to be difficult, especially in light of recent examples of infinite families which are not biquotient submersions~\cite{KS}.  We at least establish the following rigidity:
\begin{prop}\label{hol}
If $G$ is a simple compact Lie group with bi-invariant metric, and $\pi:G\ra B$ is a Riemannian submersion, then the holonomy group of $\pi$ acts transitively on the fibers.
\end{prop}

Another known rigidity phenomenon in this context is that the O'Neill tensor vanishes on horizontal zero-curvature planes~\cite{flats}.  Grove's question appears more tractable under the added assumption that the fibers are totally geodesic, in which case one expects:
\begin{conj} If $G$ is a simple compact Lie group with bi-invariant metric, and $\pi:G\ra B$ is a Riemannian submersion with connected totally geodesic fibers, then it is a left or right coset fibration; that is, $B=G/H$ or $B=H\bs G$ for some subgroup $H\subset G$, and $\pi$ is the quotient map.
\end{conj}\label{cc}
If $G$ is not simple, then the conjecture is still reasonable, but must be reworded to allow for submersions such as $\pi:G_1\times G_2\ra (H_1\bs G_1)\times(G_2/H_2)$ (where $H_i\subset G_i$), which have totally geodesic fibers, but are only equivariantly isometric to (not literally equal to) coset fibrations.  The conjecture was proven in~\cite{Ranjan}  under the added hypothesis that the fiber through the identity contains a maximal torus of $G$.

To classify Riemmanian submersions with totally geodesic fibers from any space $M$, we propose that it is useful to first understand the ``good triples'' of vectors, defined as follows:
\begin{definition}\label{D:good} Let $M$ be a Riemannian manifold and let $p\in M$. We refer to a triple $\{X,V,\mA\}\subset T_pM$ as \emph{good} if $\exp(t\cdot V(s)) = \exp(s\cdot X(t))$ for all $s,t\in\R$, where $V(s)$ denotes the Jacobi field along $s\mapsto\exp(sX)$ with $V(0)=V$ and $V'(0)=\mA$, and $X(t)$ denotes the Jacobi field along $t\mapsto\exp(tV)$ with $X(0)=X$ and $X'(0)=\mA$.  In this case, the parameterized surface $f(s,t)=\exp(t\cdot V(s)) = \exp(s\cdot X(t))$ will be called a \emph{doubly ruled parameterized surface}.
\end{definition}

%\begin{definition}\label{D:good} Let $M$ be a Riemannian manifold and let $p\in M$.  We refer to a triple $\{X,V,\mA\}\subset T_pM$ as
%\emph{bounded} if the following two Jacobi fields both have bounded norm: $V(s)$ = the Jacobi field along $s\mapsto\exp(sX)$ with %$V(0)=V$ and $V'(0)=\mA$, and $X(t)$ = the Jacobi field along $t\mapsto\exp(tV)$ with $X(0)=X$ and $X'(0)=\mA$. We call this triple
%\emph{good} if $\exp(t\cdot V(s)) = \exp(s\cdot X(t))$ for all $s,t\in\R$, in which case the parameterized surface $f(s,t)=\exp(t\cdot %V(s)) = \exp(s\cdot X(t))$ will be called a \emph{doubly ruled parameterized surface}.
%\end{definition}

A ``doubly ruled surface'' in $M$ means a surface which admits a transversal pair of smooth foliations by geodesics of $M$.  A ``doubly ruled parameterized surface'' is thus a special type of doubly ruled surface; namely, one whose geodesic foliations arise as the constant-$s$ and constant-$t$ curves of some single parameterization $(s,t)\mapsto f(s,t)$.  Since this added constraint appears restrictive, one might not expect many to exist in symmetric spaces.  We will exhibit an unexpected abundance of them in compact Lie groups, including examples beyond those which arise from coset fibrations via the following relationship between good triples and the submersion problem:

\begin{prop}\label{P:g} If $\pi:M\ra B$ is a Riemannian submersion with totally geodesic fibers, $p\in M$, and $X,V\in T_pM$ with $V$ vertical and $X$ horizontal, then $\{X,V,A(X,V)\}$ is a good triple (where $A$ denotes the O'Neill tensor).
\end{prop}

Aside from its relevance to the submersion problem, the classification of doubly ruled parameterized surfaces in symmetric spaces is perhaps a problem of independent interest.  The motivation dates back to the classical result stating that the only doubly ruled surfaces in $\R^3$ are planes, hyperbolic paraboloids and hyperboloids of one sheet (see~\cite{HCV}).  This theorem has been generalized in several directions, for example to doubly ruled submanifolds of space forms in~\cite{LF}.  It is straightforward to see that any triple $\{X,V,\mA\}\subset\R^3$ is good, and that the corresponding doubly ruled parameterized surface is a plane if $\text{dim}(\text{span}\{X,V,\mA\})=2$, or is a hyperbolic paraboloid if $\text{dim}(\text{span}\{X,V,\mA\})=3$.  The hyperboloid of one sheet is evidently not a doubly ruled parameterized surface.  Of course, one can find immersed hyperbolic paraboloids inside of flats in symmetric spaces.

The main purpose of this paper is to initiate the exploration of good triples (and hence doubly ruled parameterized surfaces) in symmetric spaces.  For compact Lie groups, we characterize the good triples as follows:
\begin{theorem}\label{classify}  Let $G$ be a compact Lie group with a bi-invariant metric.  Let $\mg$ denote the Lie algebra of $G$.  The triple $\{X,V,\mA\}\subset\mg$ is good if and only if for all integers $n,m\geq 0$, we have
$$[\ad_X^n B,\ad_V^m\overline{B}]=0, $$ where $B=\frac 12[X,V]-\mA$ and $\overline{B} = \frac 12[V,X]-\mA$.
\end{theorem}

Recall that the $A$-tensor of a left or right coset fibration is given by $$A(X,V)=\pm\frac 12[X,V]$$ for all $X\in\mg$ horizontal and $V\in\mg$ vertical (plus for a left coset fibration and minus for right), so the condition is obviously satisfied for the triple $\{X,V,A(X,V)\}$.

Using Theorem~\ref{classify}, we establish several rigidity properties for good triples.  For examples, we prove that the Jacobi fields induced by a good triple (as in Definition~\ref{D:good}) are of constant-length, provided they are of bounded-length.  Additionally, we prove that given $\{X,V\},$ there are only finitely many choices for $\mA$ for which $\{X,V,\mA\}$ is good, provided that $X$ (or $V$) satisfies a generic property which we call ``weakly regular.''

We also exhibit several examples of good triples $\{X,V,\mA\}$ in simple compact Lie groups which do not come from coset fibrations (because $\mA\neq\pm \frac 12 [X,V]$).  One example in which $X$ is not weakly regular has particularly surprising properties, which indicates that a full understanding of the good triples even in a particular compact Lie group can be subtle.

The $n=m=0$ case of the condition in Theorem~\ref{classify} says that $[[X,V], \mA]=0$.  In the final section of the paper, we prove that the analogous equation holds in more general spaces:

\begin{prop}\label{goodlocsym} If $\{X,V,\mA\}$ is a good triple in a locally symmetric space $M$, then $R(X,V)\mA=0$, where $R$ denotes the curvature tensor of $M$.
\end{prop}

It is our pleasure to thank Craig Sutton for many helpful discussions.  We also wish to thank Dartmouth College and AIM for providing hospitality.  Part of this work was completed at the AIM worship on nonnnegative curvature in September 2007.

%---------------------------------------------------------------------------------

\section{Transitive holonomy}
In this section, we prove Proposition~\ref{hol} as an application of Wilking's dual foliations work (more specifically, Theorem 3b and Proposition 6.1 of~\cite{W}).  The idea of the following proof is essentially found in~\cite{holonomy}.
\begin{proof}[Proof of Proposition~\ref{hol}]
Let $G$ be a simple compact Lie group with a bi-invariant metric.  Let $\mg$ denote its Lie algebra.  Let $\pi:G\ra B$ be a Riemannian submersion.  Denote by $\Hor_e$ and $\mV_e$ the horizontal and vertical spaces at the identity $e\in G$. Let $\tilde{H}$ denote the holonomy group of the submersion, which means the group of all diffeomorphisms of the fiber $\pi^{-1}(e)$ associated to piecewise-smooth loops in $B$ based at $\pi(e)$.

Let $\mm\subset\V_e$ denote the space of vertical vectors orthogonal to the holonomy orbit $\tilde{H}\star e$.  It will suffice to prove that $\mm=\{0\}$.

Suppose there exists a non-zero vector $V\in\mm$.  By Wilking's above-mentioned results, for all $X\in\Hor_e$, $\sigma=\text{span}\{X,V\}$ is a zero-curvature plane, so $[X,V]=0$.  Let $\alpha(t)=\exp(tX)$ and let $V(t)$ be the parallel field along $\alpha(t)$ with $V(0)=V$.  Since $\sigma$ is tangent to a totally geodesic flat in $G$, $V(t)$ is both right and left invariant along $\alpha(t)$.  In particular, $\Ad_{\alpha(t)}V=V$ for all $t$.

More generally, if $\alpha(t)$ is any horizontal piecewise geodesic path in $G$, then $\Ad_{\alpha(t)}V=V$ for all $t$.  To see this, suppose that $\{\alpha(t_1),\alpha(t_2),...\}$ are the non-smooth points of $\alpha$.  Let $V(t)$ be the parallel transport of $V$ along $\alpha(t)$.  As shown in~\cite{W}, $V(t)$ remains orthogonal to the holonomy orbits, so $\text{sec}\{Y,V(t_1)\}=0$ for all horizontal vectors $Y$, particularly for the right-derivative $Y=\alpha'(t_1^+)$.  Then $V|_{[t_1,t_2]}$ is left (respectively right) invariant along $\alpha|_{[t_1,t_2]}$ because left (respectively right) multiplication by $\alpha(t_1)^{-1}$ sends it to a parallel field along a geodesic which exponentiates to a flat, and is therefore left (respectively right) invariant.

Therefore, $\Ad_g V=V$ for all $g\in\tilde{H}\star e$.  It follows that $[Z,V]=0$ for all $Z\in\V_e$ tangent to $\tilde{H}\star e$.  Thus, $\mg=\mm\oplus(\mm^{\perp})$ is a decomposition of $\mg$ into commuting subspaces.  It follows that $\mm$ and $\mm^{\perp}$ are both ideals of $\mg$, which contradicts the hypothesis that $G$ is a simple Lie group.
\end{proof}
%---------------------------------------------------------------------------------
\section{The geometry of Riemannian submersions with totally geodesic fibers}\label{S:subs}
In this section, we summarize some facts about Riemannian submersions $\pi:M\ra B$ with totally geodesic fibers, and prove Proposition~\ref{P:g}.  Let $p\in M$.  Let $\mH_p$ and $\mV_p$ denote the horizontal and vertical spaces of $\pi$ at $p$.

We will describe how a pair $\{X,V\}$ (with $X\in\Hor_p$ and $V\in\V_p$) determines a parameterized surface in $M$.
Let $t\mapsto\gamma(t)$ denote the vertical geodesic in $M$ with $\gamma(0)=p$ and $\gamma'(0)=V$.  Let $s\mapsto\alpha_0(s)$ denote the horizontal geodesic with $\alpha_0(0)=p$ and $\alpha_0'(0)=X$.  Notice that $s\mapsto\pi(\alpha_0(s))$ is a geodesic in $B$.  For each $t\in\R$, let $s\mapsto\alpha_t(s)$ denote the horizontal lift of $s\mapsto\pi(\alpha_0(s))$ beginning at $\alpha_t(0) = \gamma(t)$.  The parameterized surface $f(s,t):=\alpha_t(s)$ has several special properties, which we discuss next.

First, by definition, this surface is ruled by horizontal geodesics.  Therefore, for any fixed $t_0\in\R$, the variational field $s\mapsto V_{t_0}(s):=\frac{\partial f}{\partial t}(s,t_0)$ is a vertical Jacobi field along the horizontal geodesic $s\mapsto f(s,t_0)=\alpha_{t_0}(s).$  These fields are often called ``holonomy Jacobi fields'' because they are related to the ``holonomy isometries,'' $h_s:\pi^{-1}(\pi(p))\ra\pi^{-1}(\pi(\alpha_0(s)))$ between the fibers.  Recall that for $q\in\pi^{-1}(\pi(p))$, $h_s(q)$ is defined as the endpoint of the horizonal lift of $\pi\circ\alpha_0|_{[0,s]}$ beginning at $q$.  The holonomy Jacobi fields record the derivatives of the holonomy isometries; that is, $d(h_s)_p(V)=V_0(s)$.  Since $h_s$ is an isometry, we learn that $V_0$ (and similarly each $V_{t_0}$) is a constant-length Jacobi field.

Second, this surface is also ruled by vertical geodesics.  That is, for each $s_0\in\R$, the path $t\mapsto f(s_0,t)$ is a vertical geodesic, because it is the image of the geodesic $t\mapsto\gamma(t)$ under the holonomy isometry $h_{s_0}$.  Therefore, for any fixed $s_0\in\R$, the variational field $t\mapsto X_{s_0}(t):=\frac{\partial f}{\partial s}(s_0,t)$ is a horizontal Jacobi field along the vertical geodesic $t\mapsto f(s_0,t).$  This field is the basic lift of the vector $(\pi\circ\alpha_0)'(s_0)$ restricted to the vertical geodesic.  Viewed this way, it is clearly a constant-length Jacobi field.  We will refer to these fields as ``basic Jacobi fields''.

The vector $X\in\Hor_p$ extends naturally to the coordinate vector field $\frac{\partial f}{\partial s}(s,t)=X_s(t)$ over the surface (we denote this extension also as $X$).  Similarly, $V\in\V_p$ extends over the surface to the other coordinate vector field $\frac{\partial f}{\partial t}(s,t)=V_t(s)$ (we denote this extension also as $V$).  The $A$-tensor determines the derivative:
$$ \nabla_XV = \nabla_VX=A(X,V),$$
which is a horizontal vector field along the surface.  In particular at $p$,
\begin{equation}\label{EE}V_0'(0) = X_0'(0) = A(X,V)\in\Hor_p.\end{equation}

Notice that the roles of $X$ and $V$ are symmetric in the sense that the parameterized surface can be defined in the following two different ways:
\begin{equation}\label{different}f(s,t) = \exp(t\cdot V_0(s)) = \exp(s\cdot X_0(t)).\end{equation}
In particular, the triple $\{X,V,A(X,V)\}$ is good, which proves Proposition~\ref{P:g}.
%---------------------------------------------------------------------------------

\section{Submersions are determined at a point}
In this section, we show that a totally geodesic foliation is completely determined by the vertical space and $A$-tensor at a single point:
\begin{prop}\label{two}
Suppose that $f_i:M\ra B_i$ ($i=1,2$) is a pair of Riemannian submersions with the same complete total space $M$, both with connected totally geodesic fibers.  If at a single point $e\in M$, the two vertical spaces agree and the two $A$-tensors agree, then the two Riemannian submersions are the same: $B_1=B_2$ and $f_1=f_2$.
\end{prop}
A more general result was proven in~\cite{Es1}, but we include the following self-contained proof:
\begin{proof}
The fibers are assumed to be connected, and are complete because $M$ is complete, so the two submersions must share a common fiber, $F$, though the point $e$, namely $F=\exp(\V_e)$.  We first show that the two horizontal distributions agree at all points of $F$.  These horizontal spaces are spanned by basic extensions to $F$ of a basis of the horizontal space at $e$.  These basic extensions are Jacobi fields along geodesics in $F$.  The derivatives of these Jacobi fields at $e$ (and thereby the entire Jacobi fields) are determined by the $A$-tensor at $e$, by Equation~\ref{EE}, and hence agree for the two submersions.

Next, fix a point $p\in F$.  At this point, the $A$-tensors for the two submersions agree because the $A$-tensor is determined by the basic extensions along $F$ of horizontal vectors at $p$, which are the same for both submersions.  It remains to show that the vertical spaces of the two submersions agree along a horizontal geodesic from $p$.  But these vertical spaces are spanned by the holonomy Jacobi fields, whose derivatives (and thereby values everywhere) are determined by the $A$-tensor at $p$, again by Equation~\ref{EE}.
\end{proof}

Now suppose $G$ is a compact Lie group with a bi-invariant metric, and $H\subset G$ is a compact Lie subgroup.  The left coset space $B=G/H$, or the right coset space $B=H\bs G$, inherits a unique metric such that the projection $\pi:G\ra B$ becomes a Riemannian submersion, called a (left or right) coset fibration.  It is easy to describe the $A$-tensor of $\pi$ at the identity $e\in G$.  For $X\in\Hor_e$ and $V\in\V_e$, we have:
$$A(X,V)=\pm(1/2)[X,V]$$
(plus for a left coset fibration, and minus for right).  The following converse holds:
\begin{cor}\label{cc}
Let $G$ be a compact Lie group with a bi-invariant metric.  Let $\pi:G\ra B$ be a Riemannian submersion with connected totally geodesic fibers.  If the $A$-tensor at the identity $e\in G$ is given by either the plus or the minus version of $A(X,V)=\pm(1/2)[X,V]$, then $\pi$ is a coset fibration.
\end{cor}
\begin{proof}
The vertical space $\V_e$ is a subalgebra because for all $U,V\in\V_e$ and all $X\in\Hor_e$,
$$\lb [U,V],X\rb = \lb [X,U],V\rb = \pm 2\lb A(X,U),V\rb = 0.$$
Let $H$ be the connected Lie subgroup of $G$ with Lie algebra $\mV_e$.  Proposition~\ref{two} implies that $\pi$ equals a fibration by cosets of $H$.
\end{proof}

%---------------------------------------------------------------------------------

\section{Bounded Jacobi fields in a compact Lie group}\label{S:jf}
Let $G$ be a compact Lie group with a bi-invariant metric.  In this section, we describe the bounded-length Jacobi fields in $G$.

Let $\mg$ denote the Lie algebra of $G$, and let $X\in\mg$.  Let $\gamma(s)=\exp(sX)$.  Decompose $\mg$ into eigenspaces of $\ad_X^2$:
$$\mg = \mU_0 \oplus\sum_{i=1}^{n} \mU_i,$$
with corresponding eigenvalues $0=\lambda_0^2>-\lambda_1^2>\cdots>-\lambda_n^2$.

For $V\in\mg$, we let $s\mapsto P_sV$ denote its parallel transport along $\gamma$.  The Jacobi fields along $\gamma$ are exactly the vector fields of the form:
\begin{equation}\label{Jacobi}J(s) = P_sE_0 + s\cdot P_sF_0 +\sum_{i=1}^n  \cos\left(\frac{\lambda_is}{2}\right)\cdot P_sE_i +\sin\left(\frac{\lambda_i s}{2}\right)\cdot P_sF_i,\end{equation}
where $E_i,F_i\in \mU_i$ for each $0\leq i\leq n$.  Equation~\ref{Jacobi} describes the unique Jacobi field along $\gamma$ with the initial data:
$$J(0)=E_0+\sum_{i=1}^n E_i,\,\,\,\text{ and }\,\,\, J'(0) = F_0 + \sum_{i=1}^n \frac{\lambda_i}{2} \cdot F_i.$$
To verify Equation~\ref{Jacobi}, it is enough to check that this vector field satisfy the Jacobi equation by using that $R$, the curvature tensor of $G$, is parallel and that $R(X,J)X = -(1/4)\ad_X^2(J)$.

 Clearly $J$ has bounded length if and only if $F_0=0$.  In this case, $J$ has constant length if and only if and for each $1\leq i\leq n$, $E_i$ and $F_i$ are orthogonal and of the same length.

The following provides a useful alternative way to describe bounded Jacobi fields in $G$:
\begin{lem}\label{LRJ} Any bounded-length Jacobi field in $G$ along $\gamma$ equals the restriction to $\gamma$ of a left-invariant vector field plus a right-invariant vector field.
\end{lem}
\begin{proof}
For each vector V in $\mU_0$, the left and right extensions of V along $\gamma$ are the same bounded Jacobi field.  For each V in $\mU_0^\perp$, the left and right extensions are different bounded Jacobi fields.  Taken together, all these fields span a space of bounded Jacobi fields of dimension equal to $\text{dim}(\mU_0) + 2\text{dim}(\mU_0^\perp)=\text{dim}(\mg)+\text{dim}(\mU_0^\perp)$.  But this is the dimension of the space of all bounded Jacobi fields along $\gamma$, since each one is determined by its initial value (in $\mg$) and its initial derivative (in $\mU_0^\perp$).
\end{proof}

In other words, if $J$ is a bounded Jacobi field along $\gamma$, then there exist ``left'' and ``right'' vectors $J_L,J_R\in\mg$ such that $J(s)=dL_{\gamma(s)}J_L + dR_{\gamma(s)}J_R$.  The pull-back of $J(s)$ to $\mg$ via left multiplication is:
\begin{equation}\label{JacobiLR}\hat{J}(s):=dL_{\gamma(s)^{-1}}J(s) = J_L + \Ad_{e^{-sX}}(J_R).\end{equation}
By the above proof, $J_R$ can be chosen in $\mU_0^\perp$, and this added condition makes the choices of $J_L$ and $J_R$ become unique.

It is useful to describe $\hat{J}(s)$ purely in terms of the Jacobi field's initial data, $V:=J(0)$ and $\mA:=J'(0)$ (rather than $J_L$ and $J_R$). For this, notice $\mA = \frac 12[X,J_L-J_R]$, so we can substitute
$$J_L=\frac 12 V + \ad_X^{-1}(\mA)\,\,\,\,\text{ and }\,\,\,\,J_R=\frac 12 V - \ad_X^{-1}(\mA)$$
into Equation~\ref{JacobiLR}, obtaining:
\begin{eqnarray}
\hat{J}(s) & = & \frac 12 V + \ad_X^{-1}(\mA) + \Ad_{e^{-sX}}\left(\frac 12 V - \ad_X^{-1}(\mA)\right)\label{JacJunk}\\
    & = & V + \sum_{n=1}^\infty\frac{(-s)^n}{n!}\left(\frac 12\ad_X^n V-\ad_X^{n-1}\mA\right)\label{JacobiVA}.
\end{eqnarray}

In Equation~\ref{JacJunk}, the expression $\ad_X^{-1}(\mA)$ is well-defined only because of our assumption that the Jacobi field has bounded length, so that $\mA$ is perpendicular to the null-space of $\ad_X^2$.  Nevertheless, it is straightforward to check that Equation~\ref{JacobiVA} is valid for bounded or unbounded Jacobi fields.  For this, simply write an unbounded Jacobi field as the sum of a bounded plus a linear Jacobi field.

\section{The Classification of Good triples in a compact Lie Group}
The goal of this section is to prove Theorem~\ref{classify}.  Let $G$ be a compact Lie group with bi-invariant metric.  Let $\mg$ denote the Lie algebra of $G$.
Let $\{X,V,\mA\}\subset\mg$.  This triple defines a parameterized surface in two ways:
\begin{enumerate}
\item Let $X(t)$ denote the Jacobi field along the geodesic $t\mapsto e^{tV}$ with $X(0)=X$ and $X'(0) = \mA$, and let $f(s,t) = \exp_{e^{tV}}(sX(t))$.
\item Let $V(s)$ denote the Jacobi field along the geodesic $s\mapsto e^{sX}$ with $V(0)=V$ and $V'(0) = \mA$, and let $\tilde{f}(s,t) = \exp_{e^{sX}}(tV(s))$.
\end{enumerate}
Recall that $\{X,V,\mA\}$ is called good if and only if $f(s,t)=\tilde{f}(s,t)$ for all $s,t\in\R$.
Denote the pull-backs to $\mg$ via left-multiplication as:
$$\hat{X}(t):=dL_{e^{-tV}}X(t)\in\mg\,\,\,\text{ and }\,\,\,\hat{V}(s):=dL_{e^{-sX}}V(s)\in\mg.$$
By Equation~\ref{JacobiVA} (with the roles of $X$ and $V$ interchanged),
\begin{equation}\label{jjjj}\hat{X}'(0) = \mA + \frac 12 \ad_X V.\end{equation}

Let $Y(s,t)\in\mg$ denote the pull-back to $\mg$ via left-multiplication of the $\ddt$ coordinate vector field of the parameterizes surface $f$; that is, $Y(s,t):=dL_{f(s,t)^{-1}} \frac{\partial f}{\partial t}(s,t)$.
Using Equations~\ref{JacobiVA} and~\ref{jjjj}, we have:
\begin{eqnarray*}
Y(s,0) = \hat{V}(s) & = & V + \sum_{n=1}^\infty\frac{(-s)^n}{n!}\left(\frac 12\ad_X^n V-\ad_X^{n-1}\mA\right)\\
                    & = & V + \sum_{n=1}^\infty\frac{(-s)^n}{n!}\left(\ad_X^n V-\ad_X^{n-1}\left(\hat{X}'(0)\right)\right).
\end{eqnarray*}
It is straightforward to generalize this to the case where ``$0$'' is replaced by an arbitrary $t\in\R$:
\begin{equation}\label{YST}Y(s,t) = V + \sum_{n=1}^\infty\frac{(-s)^n}{n!}\left(\ad_{\hat{X}(t)}^n V-\ad_{\hat{X}(t)}^{n-1}(\hat{X}'(t))\right).\end{equation}

Now assume that $\{X,V,\mA\}$ is a good triple, so that for each fixed $s$, the curve $t\mapsto f(s,t)$ is a geodesic.   In a compact Lie group with a bi-invariant metric, a geodesic's velocity field is obtained by left-translating its initial velocity vector; that is, $\frac{\partial}{\partial t}f(s,t) = dL_{f(s,t)\cdot f(s,0)^{-1}}\left(\frac{\partial}{\partial t}f(s,0)\right)$.  Therefore,
$$Y(s,t) = dL_{e^{-sX}}(V(s)) = \hat{V}(s) = Y(s,0).$$
Since the expression for $Y(s,t)$ in Equation~\ref{YST} does not depend on $t$, we learn that:
\begin{condition}\label{Q1}For each $n\geq 1$, the expression $\ad_{\hat{X}(t)}^n V-\ad_{\hat{X}(t)}^{n-1}(\hat{X}'(t))$ is independent of $t$.
\end{condition}

In fact, Condition~\ref{Q1} is equivalent to the condition that $\{X,V,\mA\}$ is good.  To see this, observe that Condition~\ref{Q1} implies that for each fixed $s$, the curve $t\mapsto f(s,t)$ is a geodesic.  By definition, $f(s,0)=\tilde{f}(s,0)$ and $\frac{\partial f}{\partial t}(s,0) = \frac{\partial \tilde{f}}{\partial t}(s,0)$, so the geodesics $t\mapsto\tilde{f}(s,t)$ and $t\mapsto f(s,t)$ must agree (as they are geodesics with the same initial position and velocity), which implies that $\tilde{f}=f$.

Next we claim that Condition~\ref{Q1} is equivalent to the following:
\begin{condition}\label{Q2}For each $m\geq 0$, the expression $\ad_{\hat{X}(t)}^m \left(\frac 12\ad_X V -\mA\right)$ is independent of $t$.\end{condition}
To demonstrate this, first re-write the expression in Condition~\ref{Q1} as:
$$\ad_{\hat{X}(t)}^n V-\ad_{\hat{X}(t)}^{n-1}(\hat{X}'(t)) = \ad_{\hat{X}(t)}^{n-1}\left(\ad_{\hat{X}(t)}V-\hat{X}'(t)\right),$$
so it remains simplify the above inner expression as:
\begin{equation}\label{alongtheway}\ad_{\hat{X}(t)}V-\hat{X}'(t) = -\ad_V\hat{X}(t) - \hat{X}'(t) = \frac 12\ad_X V -\mA.\end{equation}

To justify the second equality in Equation~\ref{alongtheway}, we use Equation~\ref{JacobiVA} (with the roles $X$ and $V$ interchanged) to express $\hat{X}(t)$ as:
\begin{equation}\label{sway}\hat{X}(t) = X + \sum_{n=1}^\infty\frac{(-t)^n}{n!}\left(\frac 12\ad_V^n X-\ad_V^{n-1}\mA\right),\end{equation}
so,
\begin{eqnarray*}
\ad_V\hat{X}(t) + \hat{X}'(t) & = & \ad_V X +\sum_{n=1}^\infty\frac{(-t)^n}{n!}\left(\frac 12\ad_V^{n+1} X-\ad_V^n\mA\right)\\
   &   & +\left(\mA-\frac 12 \ad_VX \right) + \sum_{n=1}^\infty-\frac{(-t)^n}{n!}\left(\frac 12\ad_V^{n+1} X-\ad_V^n\mA\right)\\
   & = & \mA+\frac 12\ad_V X = \mA-\frac 12 \ad_X V
\end{eqnarray*}
This verifies that Conditions~\ref{Q2} is equivalent to Condition~\ref{Q1}, and is therefore equivalent to the condition that $\{X,V,\mA\}$ is a good triple.

It remains to show that Condition~\ref{Q2} is equivalent to the condition in Theorem~\ref{classify}.  Condition~\ref{Q2} says that for each $n\geq 0$, the expression $\ad_{\hat{X}(t)}^n(B)$ is independent of $t$.  Using Equation~\ref{sway}, this implies that:
\begin{eqnarray}\label{monkey}
\ad_X^{n+1} B & = & \ad_{\hat{X}(t)}(\ad^n_{\hat{X}(t)}B) = \ad_{\hat{X}(t)}(\ad_X^n B)\\
              & = & \ad_X^{n+1}B + \left[\sum_{m=0}^\infty\frac{(-t)^{m+1}}{(m+1)!}\ad_V^m\overline{B},\ad_X^n B\right],\notag
\end{eqnarray}
so $[\ad_V^m\overline{B},\ad_X^n B]=0$ for all $n,m\geq 0$.  On the other hand, if $[\ad_V^m\overline{B},\ad_X^n B]=0$ for all $n,m\geq 0$, then the equality of the second and fourth terms of Equation~\ref{monkey} gives $\ad_{\hat{X}(t)}^{n+1}(B) = \ad_X^{n+1} B$ for all $n\geq 0$, which is exactly Condition~\ref{Q2}.
This ends the proof of Theorem~\ref{classify}.

%%-------------------------------------------------------------------------------------
\section{Examples of good triples in a compact Lie group}
In this section, we construct examples of good triples in simple compact Lie groups which do not come from coset fibrations.  The idea is to construct one-dimensional biquotient submersions which have totally geodesic fibers along certain horizontal geodesics, so that the discussion in Section~\ref{S:subs} applies along these horizontal geodesics.

Let $G$ be a simple compact Lie group with a bi-invariant metric, let $\mg$ denote its Lie algebra, and let $X\in\mg$.  With respect to a maximal abelian subalgebra of $\mg$ containing $X$, suppose that there are two roots, $\alpha$ and $\beta$, such that $\alpha\pm\beta$ is not a root and such that $\alpha(X)\neq 0$ and $\beta(X)\neq 0$.  In this case, choose vectors $V_\alpha,V_\beta\in\mg$ from the corresponding two root spaces.  Notice that
\begin{equation}\label{lamb}[\Ad_{e^{sX}}V_\alpha,V_\beta]=0\,\,\text{ for all }s\in\R.\end{equation}
This pair of vectors determines an  $\R$-action on $G$, defined so that for $g\in G$ and $t\in\R$, $t\star g:=\left(e^{tV_\alpha}\right)\cdot g\cdot \left(e^{-tV_\beta}\right)$.  The vertical space at $\mg$ of the Riemannian submersion $\pi:G\ra G/\R$ is spanned by $V:=V_\alpha-V_\beta$.  Notice that $\gamma(s):=e^{sX}$ is a horizontal geodesic.  We claim that each fiber of $\pi$ which intersects $\gamma$ is totally geodesic.

To verify that the fibers along $\gamma$ are totally geodesic, we need only check that they are geodesics.  Using Equation~\ref{lamb}, for fixed $s\in\R$, the image of the fiber through $\gamma(s)$ under the left-translation map $L_{\gamma(s)^{-1}}$ equals the path
$$t\mapsto e^{-sX}e^{tV_\alpha}e^{sX}e^{-tV_\beta} = e^{\left(\Ad_{e^{-sX}}(tV_\alpha)\right)}\cdot e^{-tV_\beta} =
 e^{t(\Ad_{-sX}V_\alpha - V_\beta)},$$
which is the geodesic in the direction of $\Ad_{-sX}V_\alpha - V_\beta$.  Since the fibers along $\gamma$ are totally geodesic, the discussion in Section~\ref{S:subs} establishes that $\{X,V,\mA\}$ is a good triple, where $\mA:=A(X,V)=-\frac 12 [X,V_\alpha+V_\beta]$.  Alternately, it is straightforward to verify that this triple is good directly from Theorem~\ref{classify}.  Since $\alpha(X)\neq 0$ and $\beta(X)\neq 0$, notice that $A(X,V)\neq \pm \frac 12 [X,V]$, which distinguishes this biquotient fibration from coset fibrations.

%%---------------------------------------------------------------------------------------
\section{An alternative characterization of good triples in a compact Lie group}
In Theorem~\ref{classify}, the condition under which $\{X,V,\mA\}$ is good is symmetric in $X$ and $V$, as expected.  However, for Riemannian submersions with totally geodesic fibers, holonomy Jacobi fields differ from basic Jacobi fields, so it is useful to break the symmetry by focusing on only one type of Jacobi field.  Therefore, we will now derive conditions under which a single bounded Jacobi field comes from a good triple. These new conditions will be used in Section 10 to construct examples of good triples with unexpected properties.

Let $X\in\mg$, and let $V(s)$ be a bounded Jacobi field along the geodesic $s\mapsto e^{sX}$.  Choose $V_L,V_R\in\mg$, with $V_R$ orthogonal to the null-space of $\ad_X^2$, such that
$$\hat{V}(s) :=dL_{e^{-sX}}V(s) = V_L+\Ad_{e^{-sX}}V_R,$$
as in Equation~\ref{JacobiLR}.  Define $V:=V(0)=V_L+V_R$ and $\mA:=V'(0)=\frac 12[X,V_L-V_R]$.  We seek conditions in terms of $\{X,V_L,V_R\}$ under which the triple $\{X,V,\mA\}$ is good.
\begin{cor}\label{LRCOR}
Let $G$ be a compact Lie group with a bi-invariant metric.  Let $X,V_L,V_R\in\mg$.  Assume that $V_R$ is orthogonal to the null-space of $\ad_X^2$.  Then the triple $$\left\{X,V_L+V_R,\frac 12[X,V_L-V_R]\right\}$$ is good if and only the following two conditions are satisfied:
\begin{enumerate}
\item $\left[\Ad_{e^{-tX}}V_R,\ad_X V_L\right]=0$ for all $t\in\R$, and
\item $\left[\ad_{V_L}^m X,\left[\Ad_{e^{-tX}}V_R,V_L\right]\right]=0$ for all $t\in\R$ and all $m\geq 1$.
\end{enumerate}
\end{cor}
The remainder of this section is devoted to the proof of Corollary~\ref{LRCOR}.  Consider a root-space decomposition, $$\mg=\tau\oplus_{i\in I}\ml_i,$$ where $\tau$ is a maximal abelian subalgebra containing $X$, and $I$ is the index set for the two-dimensional root spaces $\ml_i$.  For $W\in\mg$, let $W^i$ denote the orthogonal projection of $W$ onto $\ml_i$ and let $W^0$ denote denote the projection onto $\tau$.

\begin{lem}\label{jojo} Let $W,U\in\mg$.  Assume $W$ is orthogonal to the null-space of $\ad_X^2$.  If $[\Ad_{e^{-sX}}W,U]$ is constant in $s$, then it vanishes.
\end{lem}

\begin{proof}
Assume that
\begin{eqnarray*}
[\Ad_{e^{-sX}} W,U] & = & \left[\sum_{i\in I}\Ad_{e^{-sX}}W^i,U^0+\sum_{j\in I} U^j\right] \\
                    & = &\sum_{i\in I}[\Ad_{e^{-tX}}W^i,U^0]+\sum_{i,j\in I}\left[\Ad_{e^{-tX}}W^i,U^j\right]
\end{eqnarray*}
is constant in $s$.  The $\tau$-component of this expression comes from the diagonal $i=j$ terms of the double sum.  Each such diagonal term has the form $C_i\cos(\lambda_i + a_it)r_i$ for some $C_i,\lambda_i,a_i\in\R$, where $r_i\in\tau$ is the associated dual root.  The only way that several terms of this form could sum to a constant function is if they sum to zero.

Similarly, every term of the single-sum and every non-diagonal term of the double-sum above equals a circle in a single root space.  That is, it has the form $C\cos(\lambda + at)E + C\sin(\lambda + at)F$, where $\{E,F\}$ is an orthonormal basis of one of the $\ml$'s.  The only way that several such circles could cancel so as to sum to a constant function is if they sum to zero.
\end{proof}

Define $B,\overline{B}$ as in Theorem~\ref{classify}.  Notice that $B=[X,V_R]$ and $\overline{B}=-[X,V_L]$.  Condition~\ref{Q2} (with the roles of $X$ and $V$ reversed) says $\left\{X,V_L+V_R,\frac 12[X,V_L-V_R]\right\}$ is good if and only if the following is satisfied:
\begin{condition}\label{Q4}For each $n\geq 0$, $\ad_{\hat{V}(s)}^n(\overline{B})$ is constant in $s$.
\end{condition}
It remains to show that Condition~\ref{Q4} is equivalent to the conditions of Corollary~\ref{LRCOR}, which is done as follows:

\begin{proof}[Proof of Corollary~\ref{LRCOR}]
Assume that the triple is good.  The $n=1$ case of Condition~\ref{Q4} says that the following is constant in $s$:
$$[\hat{V}(s),\overline{B}] = \left[V_L+\Ad_{e^{-sX}} V_R,-[X,V_L]\right] = \ad_{V_L}^2X - \left[\Ad_{e^{-sX}} V_R,[X,V_L]\right].$$
Lemma~\ref{jojo} implies that the second term of this last expression vanishes, so $$\left[\Ad_{e^{-sX}} V_R,[X,V_L]\right]=0 \,\,\text{ and }\,\,[\hat{V}(s),\overline{B}] = \ad_{V_L}^2 X .$$
Similarly, the $n=2$ case of Condition~\ref{Q4} gives that
$$[\hat{V}(s),[\hat{V}(s),\overline{B}]] = \left[V_L+\Ad_{e^{-sX}} V_R,\ad_{V_L}^2X\right] = \ad_{V_L}^3X + \left[\Ad_{e^{-sX}} V_R,\ad_{V_L}^2 X\right]$$
is constant in $s$, which implies that
$$\left[\Ad_{e^{-sX}} V_R,\ad_{V_L}^2 X\right]=0 \,\,\text{ and }\,\,\ad_{\hat{V}(s)}^2\overline B = \ad_{V_L}^3 X  .$$
Continue by induction to get that for all $n\geq 1$,
$$\left[\Ad_{e^{-sX}} V_R,\ad_{V_L}^n X\right]=0 \,\,\text{ and }\,\,\ad_{\hat{V}(s)}^n\overline B = \ad_{V_L}^{n+1} X  .$$
in particular:
\begin{condition}\label{Q5}
$\left[\Ad_{e^{-sX}}V_R,\ad^n_{V_L}X\right] = 0$ for all $s\in\R$ and all $n\geq 1$.
\end{condition}
In fact the above work shows that Condition~\ref{Q5} is equivalent to Condition~\ref{Q4}, so it remains to show that Condition~\ref{Q5} is equivalent to the two conditions of Corollary~\ref{LRCOR}.  Indeed, the $n=1$ case of Condition~\ref{Q5} coincides with condition (1) of Corollary~\ref{LRCOR}.  Moreover, by the Jacobi identity, we have:
\begin{eqnarray*}
\left[\ad_{V_L}^mX,\left[\Ad_{e^{-sX}}V_R,V_L\right]\right]
  & = &
       -\left[\Ad_{e^{-sX}}V_R,\left[V_L,\ad_{V_L}^mX\right]\right]
       -\left[V_L,\left[\ad_{V_L}^mX, \Ad_{e^{-sX}}V_R\right]\right]\\
  & = &-\left[\Ad_{e^{-sX}}V_R,\ad_{V_L}^{m+1}X\right]
       +\ad_{V_L}\left[\Ad_{e^{-sX}}V_R,\ad_{V_L}^mX \right],
\end{eqnarray*}
which completes the proof.
\end{proof}

%%--------------------------------------------------------------------------
\section{Good triples induce constant length Jacobi fields}

In this section, we prove that good triples induce \emph{constant-length} Jacobi fields, provided they induce bounded-length Jacobi fields, and we also prove a related finiteness result, which requires the following definition:
\begin{definition} Let $G$ be a compact Lie group. Let $\mg$ denote its Lie algebra.  The vector $X\in\mg$ is called ``weakly regular'' if for any nonzero eigenvalue of $\ad_X^2$, the collection of roots associated to the root spaces which comprise the corresponding eigenspace is linearly independent.
\end{definition}
\begin{prop}\label{new} Let $G$ be a compact Lie group with a bi-invariant metric.  Let $\mg$ denote its Lie algebra. Let $\{X,V,\mA\}\subset\mg$ be a good triple.  Assume that the Jacobi field $V(s)$ (defined as in Definition~\ref{D:good}) has bounded length.
\begin{enumerate}
\item The Jacobi field $V(s)$ has constant-length.
\item If $X$ is weakly regular, then $\mA$ projects onto each root space (with respect to a maximal abelian subalgebra containing $X$) as plus or minus half the bracket of $X$ with the projection of $V$ onto this root space.  In particular, there are only finitely many possibilities for $\mA$ (corresponding to the possible sign choices, not all of which necessarily yield good triples).
\end{enumerate}
\end{prop}
If $X$ has the stronger regularity property that each eigenspace is comprised of a single root space, then notice that the conclusion of part (2) follows from part (1) together with our previous classification of constant length Jacobi fields.
\begin{proof}
As in Lemma~\ref{LRJ}, choose $V_L,V_R\in\mg$ such that $V=V_L+V_R$ and $\mA=\frac 12[X,V_L-V_R]$ and $V_R$ is orthogonal to the null-space of $\ad_X^2$.  To prove part (1) of the proposition, it suffices to verify that the norm of $\hat{V}(s) := V_L+\Ad_{e^{-sX}} V_R$ is independent of $s$, since this is the left-pull-back of one of the Jacobi fields determined by the good triple.

As before, consider a root-space decomposition, $\mg=\tau\oplus_{i\in I}\ml_i$, where $\tau$ is a maximal abelian subalgebra containing $X$, and $I$ is the index set for the 2-dimensional root spaces $\ml_i$.  For $W\in\mg$, let $W^i$ denote the orthogonal projection of $W$ onto $\ml_i$ and let $W^0$ denote denote the projection onto $\tau$.  Part (2) of the proposition is equivalent to the following: if $X$ is weakly regular, then for each index $i$, either $V_L^i=0$ or $V_R^i=0$.

Each eigenspace of $\ad_X^2$ equals a sum of a subset of the root spaces.  Choose a fixed such eigenspace, $\Lambda=\oplus_{i\in I'}\ml_i$, where $I'\subset I$ is the corresponding subset of indices.  For notational convenience, we assume without loss of generality that the corresponding eigvalue equals $1$.  Condition~\ref{Q5} says that for all $n\geq 1$ and all $s\in\R$, we have
$$ \sum_{i\in I}\left[\Ad_{e^{-sX}}V_R^i,\ad^n_{V_L}X\right] = 0$$
Two non-vanishing terms of this sum have the same period if and only if the corresponding two root spaces are grouped into the same eigenspace of $\ad_X^2$.  Terms with different periods cannot cancel each other, so the sum over $I'$ alone vanishes:
\begin{equation}\label{buba} \sum_{i\in I'}\left[\Ad_{e^{-sX}}V_R^i,\ad^n_{V_L}X\right] = 0.\end{equation}

Notice that $\ad_X$ induces an orientation on each $\ml_i\subset\Lambda$.  For each $i\in I'$, let $\theta_i\in[0,2\pi)$ denote the oriented angle from $V_L^i$ to $V_R^i$; that is, the smallest value of $\theta$ such that $\Ad_{e^{\theta X}}V_L^i$ is a positive scalar multiple of $V_R^i$.   Also denote $A_i:=|V_L^i|\cdot|V_R^i|$.  Finally, let $\hat{r}_i\in\tau$ denote the corresponding dual-root; that is, $\hat{r}_i = [W,\ad_X W]$ for any unit-length $W\in\ml_i$.  It is useful to re-define $I'$ by removing all indices $i$ for which $A_i=0$; this way, the angle $\theta_i$ is well-defined for each $i\in I'$.

The $n=1$ case of Equation~\ref{buba}  says:
\begin{equation}\label{char}\sum_{i\in I'}\left[\Ad_{e^{-sX}}V_R^i,\ad_{X}V_L\right] = 0.\end{equation}
In particular, the projection of this sum of brackets onto $\tau$ vanishes, so:
$$0 = \sum_{i\in I'}[\Ad_{e^{-sX}} V_R^i,\ad_X V_L^i] = \sum_{i\in I'}A_i\cos(\theta_i-s)\cdot\hat{r}_i.$$

If $X$ is weakly regular, then the $\hat{r}_i$'s are linearly independent, so each $A_i$ must vanish, which proves part (2) of the proposition.

Part (1) of the proposition is proven by taking the inner product with $X$:
\begin{equation}\label{obama}0=\sum_{i\in I'}A_i\cos(\theta_i-s)\lb\hat{r}_i,X\rb = \sum_{i\in I'}A_i\cos(\theta_i-s).\end{equation}

The condition that  $\hat{V}(s) = V_L + \Ad_{e^{-sX}} V_R$ has \emph{constant} length means that the projection of $\hat{V}(s)$ onto each such eigenspace $\Lambda$ has constant length, which is equivalent to the following being constant in $s$:
\begin{equation}\label{PK}
f(s)  = \lb V_L^\Lambda,\Ad_{e^{-sX}} V_R^\Lambda\rb
      = \sum_{i\in I'} \lb V_L^i,\Ad_{e^{-sX}} V_R^i\rb
       =   \sum_{i\in I'} A_i\cos(\theta_i-s)
\end{equation}
Comparing to Equation~\ref{obama} verifies that $\hat{V}(s)$ has constant length.
\end{proof}
%%--------------------------------------------------------------------------
\section{Another example of a good triple in a compact Lie group}
In this section, we exhibit a good triple which demonstrates that the conclusion of Part (2) of Proposition~\ref{new} can be false when $X$ is not weakly regular.  The following is immediate from Corollary~\ref{LRCOR}:
\begin{prop} If $\{X,V_L,V_R\}\subset \mg$ satisfy the following properties:
\begin{enumerate}
\item $V_L$ and $V_R$ lie in a single eigenspace of $\ad_X^2$.
\item $[Ad_{e^{-tX}} V_R,V_L]=0$ for all $t\in\R$.
\end{enumerate}
then $\left\{X,V_L+V_R,\frac 12[X,V_L-V_R]\right\}$ is a good triple.
\end{prop}

To describe our example, let $$G=SU(4) \text{ and }X=\text{diag}(-\mathbf{i},0,\mathbf{i},0)\in\mg=su(4).$$  Notice that $+1$ is an eigenvalue of $\ad_X^2$ whose corresponding eigenspace is comprised of four root spaces, commonly denoted as $\ml_{12},\ml_{23},\ml_{34},\ml_{14}$.  Let $\{E_{ij},F_{ij}\}\subset\mg$ denote the standard basis of $\ml_{ij}$; that is, $E_{ij}$ has $1$ in position $(i,j)$ and $-1$ in position $(j,i)$, while $F_{ij}$ has $\mathbf{i}$ in positions $(i,j)$ and $(j,i)$.  Define
$$V_L = E_{12} + E_{23} + E_{34} + E_{14} \text{ and } V_R = F_{12}+F_{23}+F_{34} -F_{14}.$$
We claim that $[Ad_{e^{-tX}} V_R,V_L]=0$ for all $t\in\R$, and therefore that the triple $\left\{X,V_L+V_R,\frac 12[X,V_L-V_R]\right\}$ is good.  It is sufficient to verify this at $t=0$ and $t=\pi/2$, which is a straightforward calculation.

In this example, there are root spaces onto which $V_L$ and $V_R$ both have non-zero projections, which is surprising.  In particular, this example does not satisfy the conclusion of Part (2) of Proposition~\ref{new}.

%%--------------------------------------------------
\section{Good triples in locally symmetric spaces}
In this section, we prove Proposition~\ref{goodlocsym}, which provides a necessary condition for a triple to be good in any locally symmetric space.
\begin{proof}[Proof of Proposition~\ref{goodlocsym}]
 Consider the surface $f(s,t)=\exp(t\cdot V(s)) = \exp(s\cdot X(t))$ induced by a good triple $\{X,V,\mA\}\subset T_pM$ as in Definition~\ref{D:good} and introduce the following notations: $X_{s,t}:=f_{\ast(s,t)}\frac{\partial}{\partial s}=\frac{\partial f}{\partial s}(s,t),V_{s,t}:=f_{\ast(s,t)}\frac{\partial}{\partial t}=\frac{\partial f}{\partial t}(s,t),$ and $A_{s,t}:=\nabla_{\frac{\partial}{\partial t}}X_{s,t}.$ (Notice that $X_{0,t}=X(t)$ and $V_{s,0}=V(s)$.)

 The formulas below are well known to hold for general parametrized surfaces:
 \begin{equation}\label{e1} \nabla_{\frac{\partial}{\partial t}}X_{s,t}=\nabla_{\frac{\partial}{\partial s}}V_{s,t}
 \end{equation}
 \begin{equation}\label{e2} \nabla_{\frac{\partial}{\partial s}}\nabla_{\frac{\partial}{\partial t}}Y_{s,t}-
 \nabla_{\frac{\partial}{\partial t}}\nabla_{\frac{\partial}{\partial s}}Y_{s,t}=R\left(X_{s,t},V_{s,t}\right)Y_{s,t},
 \end{equation}
 for any vector field $Y$ along $f.$

 Using ~\ref{e2} with $Y_{s,t}=X_{s,t},$ we obtain
  $$\nabla_{\frac{\partial}{\partial s}}\nabla_{\frac{\partial}{\partial t}}X_{s,t}-
 \nabla_{\frac{\partial}{\partial t}}\nabla_{\frac{\partial}{\partial s}}X_{s,t}=R\left(X_{s,t},V_{s,t}\right)X_{s,t}$$
  But $s\rightarrow f(s,t)$ is a geodesic, so  $\nabla_{\frac{\partial}{\partial s}}X_{s,t}=\nabla_{\frac{\partial}{\partial s}}\frac{\partial f}{\partial s}=0.$ This way,
  \begin{equation}\label{e3} \nabla_{\frac{\partial}{\partial s}}\nabla_{\frac{\partial}{\partial t}}X_{s,t}=R\left(X_{s,t},V_{s,t}\right)X_{s,t}.
  \end{equation}
  Similarly,
 \begin{equation}\label{e4} -\nabla_{\frac{\partial}{\partial t}}\nabla_{\frac{\partial}{\partial s}}V_{s,t}=R\left(X_{s,t},V_{s,t}\right)V_{s,t},
  \end{equation}
by using ~\ref{e2} with $Y_{s,t}=V_{s,t}.$

 Applying $\nabla_{\frac{\partial}{\partial t}}$ in ~\ref{e3}, $\nabla_{\frac{\partial}{\partial s}}$ in ~\ref{e4}, and adding the two relations yields
 \begin{equation}\label{e5}\nabla_{\frac{\partial}{\partial t}}\nabla_{\frac{\partial}{\partial s}}\nabla_{\frac{\partial}{\partial t}}X_{s,t}-\nabla_{\frac{\partial}{\partial s}}\nabla_{\frac{\partial}{\partial t}}\nabla_{\frac{\partial}{\partial s}}V_{s,t}=\nabla_{\frac{\partial}{\partial t}}R\left(X_{s,t},V_{s,t}\right)X_{s,t}+\nabla_{\frac{\partial}{\partial s}}R\left(X_{s,t},V_{s,t}\right)V_{s,t}.
 \end{equation}
 Observe that, since the manifold is locally symmetric,
\begin{eqnarray*}\nabla_{\frac{\partial}{\partial t}}R\left(X_{s,t},V_{s,t}\right)X_{s,t} & = & R\left(\nabla_{\frac{\partial}{\partial t}}X_{s,t},V_{s,t}\right)X_{s,t}\\
 & & +R\left(X_{s,t},\nabla_{\frac{\partial}{\partial t}}V_{s,t}\right)X_{s,t}+R\left(X_{s,t},V_{s,t}\right)\nabla_{\frac{\partial}{\partial t}}X_{s,t}.
 \end{eqnarray*}
 As the second term on the right side of the equality above is zero due to the fact that $t\rightarrow f(s,t)$ is a geodesic, we have
 \begin{equation}\label{e6}\nabla_{\frac{\partial}{\partial t}}R\left(X_{s,t},V_{s,t}\right)X_{s,t}=R\left(\nabla_{\frac{\partial}{\partial t}}X_{s,t},V_{s,t}\right)X_{s,t}+R\left(X_{s,t},V_{s,t}\right)A_{s,t}.
 \end{equation}
 Similarly, based on the same type of computations as above and relation ~\ref{e1},
 \begin{equation}\label{e7}\nabla_{\frac{\partial}{\partial s}}R\left(X_{s,t},V_{s,t}\right)V_{s,t}=R\left(X_{s,t},\nabla_{\frac{\partial}{\partial s}}V_{s,t}\right)X_{s,t}+R\left(X_{s,t},V_{s,t}\right)A_{s,t}.
 \end{equation}
 \noindent
 By adding relations ~\ref{e6} and ~\ref{e7}, it follows that the term on the right side of relation ~\ref{e5} is
\begin{eqnarray*}\nabla_{\frac{\partial}{\partial t}}R\left(X_{s,t},V_{s,t}\right)X_{s,t}+\nabla_{\frac{\partial}{\partial s}}R\left(X_{s,t},V_{s,t}\right)V_{s,t} &=& R(X_{s,t},A_{s,t})V_{s,t}+R(A_{s,t},V_{s,t})X_{s,t}\\ & & +2R(X_{s,t},V_{s,t})A_{s,t}.
\end{eqnarray*}
 On the other hand, by using relation ~\ref{e2} with $Y_{s,t}=\nabla_{\frac{\partial}{\partial t}}X_{s,t}=\nabla_{\frac{\partial}{\partial s}}V_{s,t}=A_{s,t},$ the term on the left side of relation ~\ref{e5} can be written as
 $$\nabla_{\frac{\partial}{\partial t}}\nabla_{\frac{\partial}{\partial s}}\nabla_{\frac{\partial}{\partial t}}X_{s,t}-\nabla_{\frac{\partial}{\partial s}}\nabla_{\frac{\partial}{\partial t}}\nabla_{\frac{\partial}{\partial s}}V_{s,t}=R(X_{s,t},V_{s,t})A_{s,t}.$$
 Based on the two formulas above, relation ~\ref{e5} becomes
 \begin{equation}\label{e8} R(X_{s,t},A_{s,t})V_{s,t}+R(A_{s,t},V_{s,t})X_{s,t}+R(X_{s,t},V_{s,t})A_{s,t}=0.
 \end{equation}
 But, by the first Bianchi identity, we have
 $$R(X_{s,t},V_{s,t})A_{s,t}+R(V_{s,t},A_{s,t})X_{s,t}+R(A_{s,t},X_{s,t})V_{s,t}=0.$$
 So,
 \begin{equation}\label{e9} R(A_{s,t},V_{s,t})X_{s,t}+R(X_{s,t},A_{s,t})V_{s,t}=R(X_{s,t},V_{s,t})A_{s,t}.
 \end{equation}
 Combining relations ~\ref{e8} and ~\ref{e9}, we obtain
\begin{equation}\label{e10}R(X_{s,t},V_{s,t})A_{s,t}=0.
\end{equation}
When $t=s=0,$ relation ~\ref{e10} becomes $R(X,V)\mA=0,$ as claimed.
\end{proof}
%%-------------------------------------------------
\bibliographystyle{amsplain}

\end{document}